\documentclass[12pt,reqno]{amsart}
\usepackage{eurosym}
\usepackage{amsfonts}
\usepackage{amsmath}
\usepackage[bookmarksnumbered, colorlinks, plainpages]{hyperref}
\usepackage{epsfig}
\usepackage{epstopdf}
\usepackage{graphicx}
\usepackage{hyperref}
\usepackage{color}
\usepackage{caption}
\usepackage{subcaption}

\hypersetup{colorlinks=true,linkcolor=red, anchorcolor=green, citecolor=cyan, urlcolor=red, filecolor=magenta, pdftoolbar=true}
\newtheorem{theorem}{Theorem}[section]

\newtheorem{proposition}[theorem]{Proposition}
\newtheorem{corollary}[theorem]{Corollary}
\theoremstyle{definition}
\newtheorem{definition}[theorem]{Definition}
\newtheorem{example}[theorem]{Example}

\theoremstyle{remark}
\newtheorem{remark}[theorem]{Remark}
\numberwithin{equation}{section}

\begin{document}
\title[A new contribution to discontinuity at fixed point]{A new
contribution to discontinuity at fixed point}
\author[N. Ta\c{s} \MakeLowercase{and} N. Y\i lmaz \"{O}zg\"{u}r]{N\.{I}HAL
Ta\c{s}$^{1}$ \MakeLowercase{and} N\.{I}HAL YILMAZ \"{O}zg\"{u}r$^{1}$}
\address{$^{1}$Bal\i kesir University, Department of Mathematics, 10145
Bal\i kesir, TURKEY.}
\email{%
\textcolor[rgb]{0.00,0.00,0.84}{nihaltas@balikesir.edu.tr;
nihal@balikesir.edu.tr}}

\begin{abstract}
The aim of this paper is to obtain new solutions to the open question on the
existence of a contractive condition which is strong enough to generate a
fixed point but which does not force the map to be continuous at the fixed
point. To do this, we use the right-hand side of the classical Rhoades'
inequality and the number $M(x,y)$ given in the definition of an $(\alpha
,\beta )$-Geraghty type-$I$ rational contractive mapping. Also we give an
application of these new results to discontinuous activation functions.
\end{abstract}

\maketitle

%\date{Received: xxxxxx; Revised: yyyyyy; Accepted: zzzzzz. \\
%\indent $^{*}$Corresponding author: N. Y\i lmaz \"{O}zg\"{u}r \\
%Bal\i kesir University, Department of Mathematics, 10145 Bal\i kesir, TURKEY
%e-mail: nihal@balikesir.edu.tr}

\setcounter{page}{1}

%\dedicatory{This paper is dedicated to Professor ABCD}

\section{\textbf{Introduction and Preliminaries}}

\label{sec:intro} Recently, some solutions to the open question on the
existence of contractive conditions which are strong enough to generate a
fixed point but which do not force the mapping to be continuous at the fixed
point has been proposed and investigated (see \cite{Bisht-2017-1}, \cite%
{Bisht-2017-2}, \cite{Kannan-1969}, \cite{Pant-1999} and \cite{Rhoades-1988}
for more details). For example, in \cite{Pant-1999}, Pant proved the
following theorem as a solution of this problem.

\begin{theorem}
\label{thm8} \cite{Pant-1999} If a self-mapping $T$ of a complete metric
space $(X,d)$ satisfies the conditions$;$

\begin{enumerate}
\item $d(Tx,Ty)\leq \phi \left( \max \left\{ d(x,Tx),d(y,Ty)\right\} \right)
$, where $\phi :%
%TCIMACRO{\U{211d} }%
%BeginExpansion
\mathbb{R}
%EndExpansion
^{+}\rightarrow
%TCIMACRO{\U{211d} }%
%BeginExpansion
\mathbb{R}
%EndExpansion
^{+}$ is such that $\phi (t)<t$ for each $t>0,$

\item For a given $\varepsilon >0$, there exists a $\delta (\varepsilon )>0$
such that%
\begin{equation*}
\varepsilon <\max \left\{ d(x,Tx),d(y,Ty)\right\} <\varepsilon +\delta \text{%
,}
\end{equation*}%
implies $d(Tx,Ty)\leq \varepsilon $,
\end{enumerate}

then $T$ has a unique fixed point $z$. Moreover, $T$ is continuous at $z$ if
and only if%
\begin{equation*}
\underset{x\rightarrow z}{\lim }\max \left\{ d(x,Tx),d(z,Tz)\right\} =0\text{%
.}
\end{equation*}
\end{theorem}

After then, in \cite{Bisht-2017-1}, Bisht and Pant obtained a new solution
of the open problem using the number%
\begin{equation*}
M(x,y)=\max \left\{ d(x,y),d(x,Tx),d(y,Ty),\frac{d(x,Ty)+d(y,Tx)}{2}\right\}
\text{.}
\end{equation*}%
Also, in \cite{Bisht-2017-2}, they proved a fixed-point theorem for this
problem using the number%
\begin{equation*}
N(x,y)=\max \left\{ d(x,y),d(x,Tx),d(y,Ty),\frac{\alpha \left[
d(x,Ty)+d(y,Tx)\right] }{2}\right\} \text{,}
\end{equation*}%
where $0\leq \alpha <1$.

Motivated by the above studies, we investigate new contractive conditions to
obtain one more solution to the open question. Before stating our main
results, we recall the following definitions which are necessary in the next
section.

\begin{definition}
\cite{Rhoades-1977} \label{def2} Let $(X,d)$ be a complete metric space and $%
T$ be a self-mapping of $X$. $T$ is called a Rhoades' mapping if the
following condition is satisfied for each $x,y\in X$, $x\neq y:$%
\begin{equation*}
d(Tx,Ty)<\max \{d(x,y),d(x,Tx),d(y,Ty),d(x,Ty),d(y,Tx)\}.
\end{equation*}
\end{definition}

Let $\Theta $ be a family of functions $\theta :\left[ 0,\infty \right)
\rightarrow \left[ 0,1\right) $ such that for any bounded sequence $\left\{
t_{n}\right\} $ of positive real numbers, $\theta \left( t_{n}\right)
\rightarrow 1$ implies $t_{n}\rightarrow 0$ and $\Phi $ be a family of
functions $\phi :\left[ 0,\infty \right) \rightarrow \left[ 0,\infty \right)
$ such that $\phi $ is continuous, strictly increasing and $\phi (0)=0$.

\begin{definition}
\cite{Chandok-2015} \label{def1} Let $(X,d)$ be a metric space, $%
T:X\rightarrow X$ be a mapping and $\alpha ,\beta :X\times X\rightarrow
%TCIMACRO{\U{211d} }%
%BeginExpansion
\mathbb{R}
%EndExpansion
^{+}$. A mapping $T$ is said to be $(\alpha ,\beta )$-Geraghty type-$I$
rational contractive mapping if there exists a $\theta \in \Theta $, such
that for all $x,y\in X$, the following condition holds$:$%
\begin{equation*}
\alpha (x,Tx)\beta (y,Ty)\phi (d(Tx,Ty))\leq \theta (\phi (M(x,y)))\phi
(M(x,y))\text{,}
\end{equation*}%
where%
\begin{equation*}
M(x,y)=\max \left\{ d(x,y),d(x,Tx),d(y,Ty),\frac{d(x,Tx)d(y,Ty)}{1+d(x,y)},%
\frac{d(x,Tx)d(y,Ty)}{1+d(Tx,Ty)}\right\}
\end{equation*}%
and $\phi \in \Phi $.
\end{definition}

On the other hand, there are some examples of self-mappings which have at
least two fixed points. In this case, new fixed-point results are necessary
for the existence of fixed points of self-mappings. Also it is important to
study the mappings with a fixed circle since there are some applications of
these kind mappings to neural networks (see \cite{Ozdemir-HNNV} for more
details). More recently, some fixed-circle theorems have been presented as a
different direction for the generalizations of the known fixed-point
theorems (see \cite{Ozgur-circle}, \cite{Ozgur-circle-S1} and \cite%
{Ozgur-circle-S2} for more details).

Now we recall the following definition of a fixed circle and one of the
known existence theorems for fixed circles.

\begin{definition}
\cite{Ozgur-circle} \label{defcircle} Let $(X,d)$ be a metric space and $%
C_{x_{0},r}=\{x\in X:d(x_{0},x)=r\}$ be a circle. For a self-mapping $%
T:X\rightarrow X$, if $Tx=x$ for every $x\in C_{x_{0},r}$ then we call the
circle $C_{x_{0},r}$ as the fixed circle of $T$.
\end{definition}

\begin{theorem}
\cite{Ozgur-circle} \label{thm7} Let $(X,d)$ be a metric space and $%
C_{x_{0},r}$ be any circle on $X$. Let us define the mapping%
\begin{equation*}
\varphi :X\rightarrow \left[ 0,\infty \right) \text{, }\varphi (x)=d(x,x_{0})%
\text{,}
\end{equation*}%
for all $x\in X$. If there exists a self-mapping $T:X\rightarrow X$
satisfying

$(C1)$ $d(x,Tx)\leq \varphi (x)-\varphi (Tx)$\newline
and

$(C2)$ $d(Tx,x_{0})\geq r$,\newline
for each $x\in C_{x_{0},r}$, then the circle $C_{x_{0},r}$ is a fixed circle
of $T$.
\end{theorem}

Our aim in this paper is to obtain new solutions to the open question on the
existence of contractive conditions which are strong enough to generate a
fixed point but which do not force the mapping to be continuous at the
point. In Section \ref{sec:1}, we use the right-hand side of the classical
Rhoades' inequality and the number $M(x,y)$ given in the definition of an $%
(\alpha ,\beta )$-Geraghty type-$I$ rational contractive mapping for this
purpose. In Section \ref{sec:2}, we give an application of these new results
to discontinuous activation functions.

\section{\textbf{Main Results}}

\label{sec:1} In this section, we investigate some contractive conditions
for the open question mentioned in the introduction.

\begin{theorem}
\label{thm1} Let $(X,d)$ be a complete metric space and $T$ be a
self-mapping on $X$ satisfying the following conditions$:$

\begin{enumerate}
\item There exists a function $\psi :%
%TCIMACRO{\U{211d} }%
%BeginExpansion
\mathbb{R}
%EndExpansion
^{+}\rightarrow
%TCIMACRO{\U{211d} }%
%BeginExpansion
\mathbb{R}
%EndExpansion
^{+}$ such that $\psi (t)<t$ for each $t>0$ and $d(Tx,Ty)\leq \psi
(M_{1}(x,y))$ where%
\begin{equation*}
M_{1}(x,y)=\max \left\{ d(x,y),d(x,Tx),d(y,Ty),\frac{d(x,Tx)d(y,Ty)}{1+d(x,y)%
},\frac{d(x,Tx)d(y,Ty)}{1+d(Tx,Ty)}\right\} ;
\end{equation*}

\item There exists a $\delta (\varepsilon )>0$ such that $\varepsilon
<M_{1}(x,y)<\varepsilon +\delta $ implies $d(Tx,Ty)\leq \varepsilon $ for a
given $\varepsilon >0$.
\end{enumerate}

Then $T$ has a unique fixed point $y_{0}\in X$ and $T^{n}x\rightarrow y_{0}$
for each $x\in X$. Also, $T$ is discontinuous at $y_{0}$ if and only if $%
\underset{x\rightarrow y_{0}}{\lim }M_{1}(x,y_{0})\neq 0$.
\end{theorem}

\begin{proof}
Let $x_{0}\in X$, $x_{0}\neq Tx_{0}$ and the sequence $\left\{ x_{n}\right\}
$ be defined as $Tx_{n}=x_{n+1}$ for all $n\in
%TCIMACRO{\U{2115} }%
%BeginExpansion
\mathbb{N}
%EndExpansion
\cup \left\{ 0\right\} $. Using the condition (1), we have%
\begin{eqnarray}
d(x_{n},x_{n+1}) &=&d(Tx_{n-1},Tx_{n})\leq \psi
(M_{1}(x_{n-1},x_{n}))<M_{1}(x_{n-1},x_{n})  \notag \\
&=&\max \left\{
\begin{array}{c}
d(x_{n-1},x_{n}),d(x_{n-1},x_{n}),d(x_{n},x_{n+1}), \\
\frac{d(x_{n-1},x_{n})d(x_{n},x_{n+1})}{1+d(x_{n-1},x_{n})},\frac{%
d(x_{n-1},x_{n})d(x_{n},x_{n+1})}{1+d(x_{n},x_{n+1})}%
\end{array}%
\right\}   \label{eqn1} \\
&=&\max \left\{ d(x_{n-1},x_{n}),d(x_{n},x_{n+1})\right\} \text{.}  \notag
\end{eqnarray}%
Assume that $d(x_{n-1},x_{n})<d(x_{n},x_{n+1})$. Then from the inequality (%
\ref{eqn1}) we get%
\begin{equation*}
d(x_{n},x_{n+1})<d(x_{n},x_{n+1})\text{,}
\end{equation*}%
which is a contradiction. So $d(x_{n},x_{n+1})<d(x_{n-1},x_{n})$ and
\begin{equation*}
M_{1}(x_{n-1},x_{n})=\max \left\{ d(x_{n-1},x_{n}),d(x_{n},x_{n+1})\right\}
=d(x_{n-1},x_{n}).
\end{equation*}
If we put $d(x_{n},x_{n+1})=u_{n}$ then from the inequality (\ref{eqn1}) we
obtain%
\begin{equation}
u_{n}<u_{n-1}\text{,}  \label{eqn2}
\end{equation}%
that is, $u_{n}$ is a strictly decreasing sequence of positive real numbers
and so the sequence $u_{n}$ tends to a limit $u\geq 0$.

Suppose that $u>0$. There exists a positive integer $k\in
%TCIMACRO{\U{2115} }%
%BeginExpansion
\mathbb{N}
%EndExpansion
$ such that $n\geq k$ implies%
\begin{equation}
u<u_{n}<u+\delta (u)\text{.}  \label{eqn3}
\end{equation}%
Using the condition (2) and the inequality (\ref{eqn2}), we get%
\begin{equation}
d(Tx_{n-1},Tx_{n})=d(x_{n},x_{n+1})=u_{n}<u\text{,}  \label{eqn4}
\end{equation}%
for $n\geq k$. The inequality (\ref{eqn4}) contradicts to the inequality (%
\ref{eqn3}). Then it should be $u=0$.

Now we show that $\left\{ u_{n}\right\} $ is a Cauchy sequence. Let us fix
an $\varepsilon >0$. Without loss of generality, we can assume that $\delta
(\varepsilon )<\varepsilon $. There exists $k\in
%TCIMACRO{\U{2115} }%
%BeginExpansion
\mathbb{N}
%EndExpansion
$ such that%
\begin{equation*}
d(x_{n},x_{n+1})=u_{n}<\delta \text{ }(0<\delta <1)\text{,}
\end{equation*}%
for $n\geq k$ since $u_{n}\rightarrow 0$. Following Jachymski (see \cite%
{Jachymski-1994} and \cite{Jachymski-1995} for more details), using the
mathematical induction, we prove%
\begin{equation}
d(x_{k},x_{k+n})<\varepsilon +\delta \text{,}  \label{eqn5}
\end{equation}%
for any $n\in
%TCIMACRO{\U{2115} }%
%BeginExpansion
\mathbb{N}
%EndExpansion
$. The inequality (\ref{eqn5}) holds for $n=1$ since%
\begin{equation*}
d(x_{k},x_{k+1})=u_{k}<\delta <\varepsilon +\delta \text{.}
\end{equation*}%
Assume that the inequality (\ref{eqn5}) is true for some $n$. We prove it
for $n+1$. Using the triangle inequality, we obtain
\begin{equation*}
d(x_{k},x_{k+n+1})\leq d(x_{k},x_{k+1})+d(x_{k+1},x_{k+n+1})\text{.}
\end{equation*}%
It suffices to show $d(x_{k+1},x_{k+n+1})\leq \varepsilon $. To do this, we
prove $M_{1}(x_{k},x_{k+n})\leq \varepsilon +\delta $, where%
\begin{equation}
M_{1}(x_{k},x_{k+n})=\max \left\{
\begin{array}{c}
d(x_{k},x_{k+n}),d(x_{k},Tx_{k}),d(x_{k+n},Tx_{k+n}), \\
\frac{d(x_{k},Tx_{k})d(x_{k+n},Tx_{k+n})}{1+d(x_{k},x_{k+n})},\frac{%
d(x_{k},Tx_{k})d(x_{k+n},Tx_{k+n})}{1+d(Tx_{k},Tx_{k+n})}%
\end{array}%
\right\} \text{.}  \label{eqn6}
\end{equation}%
Using the mathematical induction hypothesis, we find%
\begin{equation}
\begin{array}{l}
d(x_{k},x_{k+n})<\varepsilon +\delta \text{,} \\
d(x_{k},x_{k+1})<\delta \text{,} \\
d(x_{k+n},x_{k+n+1})<\delta \text{,} \\
\frac{d(x_{k},Tx_{k})d(x_{k+n},Tx_{k+n})}{1+d(x_{k},x_{k+n})}<\frac{\delta
^{2}}{1+d(x_{k},x_{k+n})}\text{,} \\
\frac{d(x_{k},Tx_{k})d(x_{k+n},Tx_{k+n})}{1+d(Tx_{k},Tx_{k+n})}<\frac{\delta
^{2}}{1+d(Tx_{k},Tx_{k+n})}\text{.}%
\end{array}
\label{eqn7}
\end{equation}%
Using the conditions (\ref{eqn6}) and (\ref{eqn7}), we have $%
M_{1}(x_{k},x_{k+n})<\varepsilon +\delta $. From the condition (2), we obtain%
\begin{equation*}
d(Tx_{k},Tx_{k+n})=d(x_{k+1},x_{k+n+1})\leq \varepsilon \text{.}
\end{equation*}%
Therefore, the inequality (\ref{eqn5}) implies that $\left\{ x_{n}\right\} $
is Cauchy. Since $(X,d)$ is a complete metric space, there exists a point $%
y_{0}\in X$ such that $x_{n}\rightarrow y_{0}$ as $n\rightarrow \infty $.
Also we get $Tx_{n}\rightarrow y_{0}$.

Now we show that $Ty_{0}=y_{0}$. On the contrary, suppose that $y_{0}$ is
not a fixed point of $T$, that is, $Ty_{0}\neq y_{0}$. Then using the
condition (1), we get%
\begin{eqnarray*}
d(Ty_{0},Tx_{n}) &\leq &\psi (M_{1}(y_{0},x_{n}))<M_{1}(y_{0},x_{n}) \\
&=&\max \left\{
\begin{array}{c}
d(y_{0},x_{n}),d(y_{0},Ty_{0}),d(x_{n},Tx_{n}), \\
\frac{d(y_{0},Ty_{0})d(x_{n},Tx_{n})}{1+d(y_{0},x_{n})},\frac{%
d(y_{0},Ty_{0})d(x_{n},Tx_{n})}{1+d(Ty_{0},Tx_{n})}%
\end{array}%
\right\}
\end{eqnarray*}%
and so taking limit for $n\rightarrow \infty $ we have%
\begin{equation*}
d(Ty_{0},y_{0})<d(y_{0},Ty_{0})=d(Ty_{0},y_{0})\text{,}
\end{equation*}%
which is a contradiction. Thus $y_{0}$ is a fixed point of $T$. We prove
that the fixed point $y_{0}$ is unique. Let $z_{0}$ be another fixed point
of $T$ such that $y_{0}\neq z_{0}$. By the condition (1), we find%
\begin{eqnarray*}
d(Ty_{0},Tz_{0}) &=&d(y_{0},z_{0})\leq \psi
(M_{1}(y_{0},z_{0}))<M_{1}(y_{0},z_{0}) \\
&=&\max \left\{
\begin{array}{c}
d(y_{0},z_{0}),d(y_{0},y_{0}),d(z_{0},z_{0}), \\
\frac{d(y_{0},y_{0})d(z_{0},z_{0})}{1+d(y_{0},z_{0})},\frac{%
d(y_{0},y_{0})d(z_{0},z_{0})}{1+d(y_{0},z_{0})}%
\end{array}%
\right\} \\
&=&d(y_{0},z_{0})\text{,}
\end{eqnarray*}%
which is a contradiction. Hence $y_{0}$ is the unique fixed point of $T$.

Finally, we prove that $T$ is discontinuous at $y_{0}$ if and only if $%
\underset{x\rightarrow y_{0}}{\lim }M_{1}(x,y_{0})\neq 0$. To do this, we
show that $T$ is continuous at $y_{0}$ if and only if $\underset{%
x\rightarrow y_{0}}{\lim }M_{1}(x,y_{0})=0$. Let $T$ be continuous at the
fixed point $y_{0}$ and $x_{n}\rightarrow y_{0}$. Then $Tx_{n}\rightarrow
Ty_{0}=y_{0}$ and%
\begin{equation*}
d(x_{n},Tx_{n})\leq d(x_{n},y_{0})+d(Tx_{n},y_{0})\rightarrow 0\text{.}
\end{equation*}%
Hence we get $\underset{n}{\lim }M_{1}(x_{n},y_{0})=0$. On the other hand,
if $\underset{x_{n}\rightarrow y_{0}}{\lim }M_{1}(x_{n},y_{0})=0$ then $%
d(x_{n},Tx_{n})\rightarrow 0$ as $x_{n}\rightarrow y_{0}$. This implies $%
Tx_{n}\rightarrow y_{0}=Ty_{0}$, that is, $T$ is continuous at $y_{0}$.
\end{proof}

\begin{remark}
Notice that the conditions $(1)$ and $(2)$ are not independent in Theorem %
\ref{thm1}. Indeed, in the cases where the condition $(2)$ is satisfied, we
obtain $d(Tx,Ty)<M_{1}(x,y)$, where $M_{1}(x,y)>0$. If $M_{1}(x,y)=0$ then $%
d(Tx,Ty)=0$. So the inequality $d(Tx,Ty)\leq \varepsilon $ holds for any $%
x,y\in X$ with $\varepsilon <M_{1}(x,y)<\varepsilon +\delta $.
\end{remark}

In the following example, we see that a self-mapping satisfying the
conditions of Theorem \ref{thm1} has a unique fixed point at which $T$ is
discontinuous.

\begin{example}
\label{exm1} Let $X=\left[ 0,4\right] $ and $d$ be the usual metric on $X$.
Let us define a self-mapping $T:X\rightarrow X$ by%
\begin{equation*}
Tx=\left\{
\begin{array}{ccc}
2 & ; & x\leq 2 \\
0 & ; & x>2%
\end{array}%
\right. \text{.}
\end{equation*}%
Then $T$ satisfies the conditions of Theorem \ref{thm1} and has a unique
fixed point $x=2$ at which $T$ is discontinuous. It can be verified in this
example that%
\begin{equation*}
d(Tx,Ty)=0\text{ and }0<M_{1}(x,y)\leq 4\text{ when }x,y\leq 2\text{,}
\end{equation*}%
\begin{equation*}
d(Tx,Ty)=0\text{ and }2<M_{1}(x,y)\leq 16\text{ when }x,y>2\text{,}
\end{equation*}%
\begin{equation*}
d(Tx,Ty)=2\text{ and }2<M_{1}(x,y)\leq 4\text{ when }x\leq 2\text{, }y>2
\end{equation*}%
and%
\begin{equation*}
d(Tx,Ty)=2\text{ and }2<M_{1}(x,y)\leq 4\text{ when }x>2\text{, }y\leq 2%
\text{.}
\end{equation*}%
Therefore the self-mapping $T$ satisfies the condition $(1)$ given in
Theorem \ref{thm1} with%
\begin{equation*}
\psi (t)=\left\{
\begin{array}{ccc}
2 & ; & t>2 \\
\frac{t}{2} & ; & t\leq 2%
\end{array}%
\right. \text{.}
\end{equation*}%
Also $T$ satisfies the condition $(2)$ given in Theorem \ref{thm1} with%
\begin{equation*}
\delta (\varepsilon )=\left\{
\begin{array}{ccc}
15 & ; & \varepsilon \geq 2 \\
5-\varepsilon & ; & \varepsilon <2%
\end{array}%
\right. \text{.}
\end{equation*}%
It can be easily checked that%
\begin{equation*}
\underset{x\rightarrow 2}{\lim }M_{1}(x,2)\neq 0\text{.}
\end{equation*}%
Consequently, $T$ is discontinuous at the fixed point $x=2$.
\end{example}

Now we give the following corollaries as the results of Theorem \ref{thm1}.

\begin{corollary}
\label{cor1} Let $(X,d)$ be a complete metric space and $T$ be a
self-mapping on $X$ satisfying the following conditions$:$

\begin{enumerate}
\item $d(Tx,Ty)\leq M_{1}(x,y)$ for any $x,y\in X$ with $M_{1}(x,y)>0;$

\item There exists a $\delta (\varepsilon )>0$ such that $\varepsilon
<M_{1}(x,y)<\varepsilon +\delta $ implies $d(Tx,Ty)\leq \varepsilon $ for a
given $\varepsilon >0$.
\end{enumerate}

Then $T$ has a unique fixed point $y_{0}\in X$ and $T^{n}x\rightarrow y_{0}$
for each $x\in X$. Also, $T$ is discontinuous at $y_{0}$ if and only if $%
\underset{x\rightarrow y_{0}}{\lim }M_{1}(x,y_{0})\neq 0$.
\end{corollary}

\begin{corollary}
\label{cor2} Let $(X,d)$ be a complete metric space and $T$ be a
self-mapping on $X$ satisfying the following conditions$:$

\begin{enumerate}
\item There exists a function $\psi :%
%TCIMACRO{\U{211d} }%
%BeginExpansion
\mathbb{R}
%EndExpansion
^{+}\rightarrow
%TCIMACRO{\U{211d} }%
%BeginExpansion
\mathbb{R}
%EndExpansion
^{+}$ such that $\psi (d(x,y))<d(x,y)$ and $d(Tx,Ty)\leq \psi (d(x,y));$

\item There exists a $\delta (\varepsilon )>0$ such that $\varepsilon
<t<\varepsilon +\delta $ implies $\psi (t)\leq \varepsilon $ for any $t>0$
and a given $\varepsilon >0$.
\end{enumerate}

Then $T$ has a unique fixed point $y_{0}\in X$ and $T^{n}x\rightarrow y_{0}$
for each $x\in X$.
\end{corollary}

In the following theorem, we see that the power contraction of the type $%
M_{1}(x,y)$ allows the possibility of discontinuity at the fixed point.

\begin{theorem}
\label{thm4} Let $(X,d)$ be a complete metric space and $T$ be a
self-mapping on $X$ satisfying the following conditions$:$

\begin{enumerate}
\item There exists a function $\psi :%
%TCIMACRO{\U{211d} }%
%BeginExpansion
\mathbb{R}
%EndExpansion
^{+}\rightarrow
%TCIMACRO{\U{211d} }%
%BeginExpansion
\mathbb{R}
%EndExpansion
^{+}$ such that $\psi (t)<t$ for each $t>0$ and $d(T^{m}x,T^{m}y)\leq \psi
(M_{1}^{\ast }(x,y))$ where
\begin{equation*}
M_{1}^{\ast }(x,y)=\max \left\{
\begin{array}{c}
d(x,y),d(x,T^{m}x),d(y,T^{m}y), \\
\frac{d(x,T^{m}x)d(y,T^{m}y)}{1+d(x,y)},\frac{d(x,T^{m}x)d(y,T^{m}y)}{%
1+d(T^{m}x,T^{m}y)}%
\end{array}%
\right\} ;
\end{equation*}

\item There exists a $\delta (\varepsilon )>0$ such that $\varepsilon
<M_{1}^{\ast }(x,y)<\varepsilon +\delta $ implies $d(T^{m}x,T^{m}y)\leq
\varepsilon $ for a given $\varepsilon >0.$
\end{enumerate}

Then $T$ has a unique fixed point. Also, $T$ is discontinuous at $y_{0}$ if
and only if $\underset{x\rightarrow y_{0}}{\lim }M_{1}^{\ast }(x,y_{0})\neq
0 $.
\end{theorem}

\begin{proof}
Using Theorem \ref{thm1}, we see that the function $T^{m}$ has a unique
fixed point $y_{0}$, that is, $T^{m}y_{0}=y_{0}$. Hence we get%
\begin{equation*}
Ty_{0}=TT^{m}y_{0}=T^{m}Ty_{0}
\end{equation*}%
and so $Ty_{0}$ is a fixed point of $T^{m}$. From the uniqueness of the
fixed point, then we obtain $Ty_{0}=y_{0}$. Consequently, $T$ has a unique
fixed point.
\end{proof}

\begin{remark}
\label{rem1} Using the continuity of the self-mapping $T^{2}$ $($resp. the
continuity of the self-mapping $T^{p}$, the orbitally continuity of the
self-mapping $T)$ and the number $M_{1}(x,y)$, we can also give new
fixed-point results for this open question $($see \cite{Bisht-2017-1} and
\cite{Bisht-2017-2} for this approach$)$.
\end{remark}

We give another result of discontinuity at fixed point on a metric space.

\begin{theorem}
\label{thm2} Let $(X,d)$ be a complete metric space and $T$ be a
self-mapping on $X$ satisfying the following conditions$:$

\begin{enumerate}
\item There exists a function $\psi :%
%TCIMACRO{\U{211d} }%
%BeginExpansion
\mathbb{R}
%EndExpansion
^{+}\rightarrow
%TCIMACRO{\U{211d} }%
%BeginExpansion
\mathbb{R}
%EndExpansion
^{+}$ such that $\psi (t)<t$ for each $t>0$ and $d(Tx,Ty)\leq \frac{1}{2}%
\psi (M_{2}(x,y))$ where
\begin{equation*}
M_{2}(x,y)=\max \left\{ d(x,y),d(Tx,x),d(Ty,y),d(Tx,y),d(Ty,x)\right\} ;
\end{equation*}

\item There exists a $\delta (\varepsilon )>0$ such that $\varepsilon
<M_{2}(x,y)<\varepsilon +\delta $ implies $d(Tx,Ty)\leq \varepsilon $ for a
given $\varepsilon >0$.
\end{enumerate}

Then $T$ has a unique fixed point $y_{0}\in X$ and $T^{n}x\rightarrow y_{0}$
for each $x\in X$. Also, $T$ is discontinuous at $y_{0}$ if and only if $%
\underset{x\rightarrow y_{0}}{\lim }M_{2}(x,y_{0})\neq 0$.
\end{theorem}

\begin{proof}
Let $x_{0}\in X$, $x_{0}\neq Tx_{0}$ and a sequence $\left\{ x_{n}\right\} $
be defined as $T^{n}x_{0}=Tx_{n}=x_{n+1}$ for all $n\in
%TCIMACRO{\U{2115} }%
%BeginExpansion
\mathbb{N}
%EndExpansion
\cup \left\{ 0\right\} $. Using the condition (1), we have%
\begin{eqnarray*}
d(x_{n},x_{n+1}) &=&d(Tx_{n-1},Tx_{n})\leq \frac{1}{2}\psi
(M_{2}(x_{n-1},x_{n}))<\frac{1}{2}M_{2}(x_{n-1},x_{n}) \\
&=&\frac{1}{2}\max \left\{
\begin{array}{c}
d(x_{n-1},x_{n}),d(x_{n},x_{n-1}),d(x_{n+1},x_{n}), \\
d(x_{n},x_{n}),d(x_{n+1},x_{n-1})%
\end{array}%
\right\} \\
&=&\frac{1}{2}\max \left\{
d(x_{n-1},x_{n}),d(x_{n+1},x_{n}),d(x_{n+1},x_{n-1})\right\} \\
&<&\frac{1}{2}\max \left\{
\begin{array}{c}
d(x_{n-1},x_{n})+d(x_{n+1},x_{n}),d(x_{n+1},x_{n})+d(x_{n-1},x_{n}), \\
d(x_{n+1},x_{n})+d(x_{n},x_{n-1})%
\end{array}%
\right\} \\
&=&\frac{1}{2}\left[ d(x_{n-1},x_{n})+d(x_{n+1},x_{n})\right]
\end{eqnarray*}%
and so%
\begin{equation}
2d(x_{n},x_{n+1})<d(x_{n-1},x_{n})+d(x_{n+1},x_{n})\text{.}  \label{eqn8}
\end{equation}%
Using the inequality (\ref{eqn8}), we get%
\begin{equation*}
d(x_{n},x_{n+1})<d(x_{n-1},x_{n})\text{.}
\end{equation*}%
If we put $d(x_{n},x_{n+1})=u_{n}$ then from the above inequality we obtain%
\begin{equation}
u_{n}<u_{n-1}\text{,}  \label{eqn9}
\end{equation}%
that is, $u_{n}$ is a strictly decreasing sequence of positive real numbers
and so the sequence $u_{n}$ tends to a limit $u\geq 0$.

Suppose that $u>0$. There exists a positive integer $k\in
%TCIMACRO{\U{2115} }%
%BeginExpansion
\mathbb{N}
%EndExpansion
$ such that $n\geq k$ implies%
\begin{equation}
u<u_{n}<u+\delta (u)\text{.}  \label{eqn10}
\end{equation}%
Using the condition (2) and the inequality (\ref{eqn9}), we get%
\begin{equation}
d(Tx_{n-1},Tx_{n})=d(x_{n},x_{n+1})=u_{n}<u\text{,}  \label{eqn11}
\end{equation}%
for $n\geq k$. The inequality (\ref{eqn11}) contradicts to the inequality (%
\ref{eqn10}). Thus it should be $u=0$.

Now we show that $\left\{ u_{n}\right\} $ is a Cauchy sequence. Let us fix
an $\varepsilon >0$. Without loss of generality, we can assume that $\delta
(\varepsilon )<\varepsilon $. There exists $k\in
%TCIMACRO{\U{2115} }%
%BeginExpansion
\mathbb{N}
%EndExpansion
$ such that%
\begin{equation*}
d(x_{n},x_{n+1})=u_{n}<\frac{\delta }{2}\text{,}
\end{equation*}%
for $n\geq k$ since $u_{n}\rightarrow 0$. Following Jachymski (see \cite%
{Jachymski-1994} and \cite{Jachymski-1995} for more details), using the
mathematical induction, we prove%
\begin{equation}
d(x_{k},x_{k+n})<\varepsilon +\frac{\delta }{2}\text{,}  \label{eqn12}
\end{equation}%
for any $n\in
%TCIMACRO{\U{2115} }%
%BeginExpansion
\mathbb{N}
%EndExpansion
$. The inequality (\ref{eqn12}) holds for $n=1$ since%
\begin{equation*}
d(x_{k},x_{k+1})=u_{k}<\frac{\delta }{2}<\varepsilon +\frac{\delta }{2}\text{%
.}
\end{equation*}%
Assume that the inequality (\ref{eqn12}) is true for some $n$. We prove it
for $n+1$. Using the triangle inequality, we have%
\begin{equation*}
d(x_{k},x_{k+n+1})\leq d(x_{k},x_{k+1})+d(x_{k+1},x_{k+n+1})\text{.}
\end{equation*}%
It suffices to show $d(x_{k+1},x_{k+n+1})\leq \varepsilon $. To do this, we
prove $M_{2}(x_{k},x_{k+n})\leq \varepsilon +\delta $, where%
\begin{eqnarray}
M_{2}(x_{k},x_{k+n}) &=&\max \left\{
\begin{array}{c}
d(x_{k},x_{k+n}),d(Tx_{k},x_{k}),d(Tx_{k+n},x_{k+n}), \\
d(Tx_{k},x_{k+n}),d(Tx_{k+n},x_{k})%
\end{array}%
\right\}  \notag \\
&=&\max \left\{
\begin{array}{c}
d(x_{k},x_{k+n}),d(x_{k+1},x_{k}),d(x_{k+n+1},x_{k+n}), \\
d(x_{k+1},x_{k+n}),d(x_{k+n+1},x_{k})%
\end{array}%
\right\}  \label{eqn13} \\
&\leq &\max \left\{
\begin{array}{c}
d(x_{k},x_{k+n}),d(x_{k},x_{k+1}),d(x_{k+n},x_{k+n+1}), \\
d(x_{k},x_{k+1})+d(x_{k},x_{k+n}),d(x_{k+n},x_{k+n+1})+d(x_{k},x_{k+n})%
\end{array}%
\right\} \text{.}  \notag
\end{eqnarray}%
Using the mathematical induction hypothesis, we get%
\begin{equation}
\begin{array}{l}
d(x_{k},x_{k+n})<\varepsilon +\frac{\delta }{2}\text{,} \\
d(x_{k},x_{k+1})<\frac{\delta }{2}\text{,} \\
d(x_{k+n},x_{k+n+1})<\frac{\delta }{2}\text{,} \\
d(x_{k},x_{k+1})+d(x_{k},x_{k+n})<\varepsilon +\delta \text{,} \\
\frac{d(x_{k+n},x_{k+n+1})+d(x_{k},x_{k+n})}{2}<\varepsilon +\delta \text{.}%
\end{array}
\label{eqn14}
\end{equation}%
Using the conditions (\ref{eqn13}) and (\ref{eqn14}), we have $%
M_{2}(x_{k},x_{k+n})<\varepsilon +\delta $. From the condition (2), we obtain%
\begin{equation*}
d(Tx_{k},Tx_{k+n})=d(x_{k+1},x_{k+n+1})\leq \varepsilon \text{.}
\end{equation*}%
Therefore, the inequality (\ref{eqn12}) implies that $\left\{ x_{n}\right\} $
is Cauchy. Since $(X,d)$ is a complete metric space, there exists a point $%
y_{0}\in X$ such that $x_{n}\rightarrow y_{0}$ as $n\rightarrow \infty $.
Also we get $Tx_{n}\rightarrow y_{0}$.

Now we show that $Ty_{0}=y_{0}$. On the contrary, $y_{0}$ is not a fixed
point of $T$, that is, $Ty_{0}\neq y_{0}$. Then using the condition (1), we
get%
\begin{eqnarray*}
d(Ty_{0},Tx_{n}) &\leq &\frac{1}{2}\psi (M_{2}(y_{0},x_{n}))<\frac{1}{2}%
M_{2}(y_{0},x_{n}) \\
&=&\frac{1}{2}\max \left\{
\begin{array}{c}
d(y_{0},x_{n}),d(Ty_{0},y_{0}),d(Tx_{n},x_{n}), \\
d(Ty_{0},x_{n}),d(Tx_{n},y_{0})%
\end{array}%
\right\}
\end{eqnarray*}%
and so taking limit for $n\rightarrow \infty $ we have%
\begin{equation*}
d(Ty_{0},y_{0})<\frac{1}{2}d(Ty_{0},y_{0})\text{,}
\end{equation*}%
which is a contradiction. Thus $y_{0}$ is a fixed point of $T$. We prove
that the fixed point $y_{0}$ is unique. Let $z_{0}$ be another fixed point
of $T$ such that $y_{0}\neq z_{0}$. From the condition (1), we find%
\begin{eqnarray*}
d(Ty_{0},Tz_{0}) &=&d(y_{0},z_{0})\leq \frac{1}{2}\psi (M_{2}(y_{0},z_{0}))<%
\frac{1}{2}M_{2}(y_{0},z_{0}) \\
&=&\frac{1}{2}\max \left\{
\begin{array}{c}
d(y_{0},z_{0}),d(y_{0},y_{0}),d(z_{0},z_{0}), \\
d(y_{0},z_{0}),d(z_{0},y_{0})%
\end{array}%
\right\} \\
&=&\frac{1}{2}d(y_{0},z_{0})\text{,}
\end{eqnarray*}%
which is a contradiction. Hence $y_{0}$ is a unique fixed point of $T$.

Finally, we prove that $T$ is discontinuous at $y_{0}$ if and only if $%
\underset{x\rightarrow y_{0}}{\lim }M_{2}(x,y_{0})\neq 0$. To do this, we
show that $T$ is continuous at $y_{0}$ if and only if $\underset{%
x\rightarrow y_{0}}{\lim }M_{2}(x,y_{0})=0$. Let $T$ be continuous at the
fixed point $y_{0}$ and $x_{n}\rightarrow y_{0}$. Then $Tx_{n}\rightarrow
Ty_{0}=y_{0}$ and%
\begin{equation*}
d(x_{n},Tx_{n})\leq d(x_{n},y_{0})+d(Tx_{n},y_{0})\rightarrow 0\text{.}
\end{equation*}%
Hence we get $\underset{n}{\lim }M_{2}(x_{n},y_{0})=0$. On the other hand,
if $\underset{x_{n}\rightarrow y_{0}}{\lim }M_{2}(x_{n},y_{0})=0$ then $%
d(x_{n},Tx_{n})\rightarrow 0$ as $x_{n}\rightarrow y_{0}$. This implies $%
Tx_{n}\rightarrow y_{0}=Ty_{0}$, that is, $T$ is continuous at $y_{0}$.
\end{proof}

In the following example, we see that the self-mapping satisfying the
conditions of Theorem \ref{thm2} has a unique fixed point at which $T$ is
continuous.

\begin{example}
\label{exm2} Let $X=\left[ 0,4\right] $ and $d$ be the usual metric on $X$.
Let us define a self-mapping $T:X\rightarrow X$ by%
\begin{equation*}
Tx=2\text{,}
\end{equation*}%
for all $x\in X$. Then $T$ satisfies the conditions of Theorem \ref{thm2}
and has a unique fixed point $x=2$ at which $T$ is continuous. It can be
verified in this example that%
\begin{equation*}
d(Tx,Ty)=0\text{ and }0\leq M_{2}(x,y)\leq 4\text{ when }x,y\in X\text{.}
\end{equation*}%
Therefore the self-mapping $T$ satisfies the condition $(1)$ given in
Theorem \ref{thm2} with%
\begin{equation*}
\psi (t)=\frac{t}{2}\text{.}
\end{equation*}%
Also $T$ satisfies the condition $(2)$ given in Theorem \ref{thm2} with%
\begin{equation*}
\delta (\varepsilon )=5-\varepsilon \text{.}
\end{equation*}%
It can be easily checked that%
\begin{equation*}
\underset{x\rightarrow 2}{\lim }M_{2}(x,2)=0\text{.}
\end{equation*}%
Consequently, $T$ is continuous at the fixed point $x=2$.
\end{example}

Now we give the following corollaries as the results of Theorem \ref{thm2}.

\begin{corollary}
\label{cor3} Let $(X,d)$ be a complete metric space and $T$ be a
self-mapping on $X$ satisfying the following conditions$:$

\begin{enumerate}
\item $d(Tx,Ty)\leq \frac{M_{2}(x,y)}{2}$ for any $x,y\in X$ with $%
M_{2}(x,y)>0;$

\item There exists a $\delta (\varepsilon )>0$ such that $\varepsilon
<M_{2}(x,y)<\varepsilon +\delta $ implies $d(Tx,Ty)\leq \varepsilon $ for a
given $\varepsilon >0$.
\end{enumerate}

Then $T$ has a unique fixed point $y_{0}\in X$ and $T^{n}x\rightarrow y_{0}$
for each $x\in X$. Also, $T$ is discontinuous at $y_{0}$ if and only if $%
\underset{x\rightarrow y_{0}}{\lim }M_{2}(x,y_{0})\neq 0$.
\end{corollary}

\begin{corollary}
\label{cor4} Let $(X,d)$ be a complete metric space and $T$ be a
self-mapping on $X$ satisfying the following conditions$:$

\begin{enumerate}
\item There exists a function $\psi :%
%TCIMACRO{\U{211d} }%
%BeginExpansion
\mathbb{R}
%EndExpansion
^{+}\rightarrow
%TCIMACRO{\U{211d} }%
%BeginExpansion
\mathbb{R}
%EndExpansion
^{+}$ such that $\psi (d(x,y))<d(x,y)$ and $d(Tx,Ty)\leq \frac{1}{2}\psi
(d(x,y))$ $;$

\item There exists a $\delta (\varepsilon )>0$ such that $\varepsilon
<t<\varepsilon +\delta $ implies $\psi (t)\leq \varepsilon $ for any $t>0$
and a given $\varepsilon >0$.
\end{enumerate}

Then $T$ has a unique fixed point $y_{0}\in X$ and $T^{n}x\rightarrow y_{0}$
for each $x\in X$.
\end{corollary}

In the following theorem, we can see that the power contraction of the type $%
M_{2}(x,y)$ allows the possibility of discontinuity at the fixed point.

\begin{theorem}
\label{thm6} Let $(X,d)$ be a complete metric space and $T$ be a
self-mapping on $X$ satisfying the following conditions$:$

\begin{enumerate}
\item There exists a function $\psi :%
%TCIMACRO{\U{211d} }%
%BeginExpansion
\mathbb{R}
%EndExpansion
^{+}\rightarrow
%TCIMACRO{\U{211d} }%
%BeginExpansion
\mathbb{R}
%EndExpansion
^{+}$ such that $\psi (t)<t$ for each $t>0$ and $d(T^{m}x,T^{m}y)\leq \frac{1%
}{2}\psi (M_{2}^{\ast }(x,y))$ where%
\begin{equation*}
M_{2}^{\ast }(x,y)=\max \left\{
d(x,y),d(T^{m}x,x),d(T^{m}y,y),d(T^{m}x,y),d(T^{m}y,x)\right\} ;
\end{equation*}

\item There exists a $\delta (\varepsilon )>0$ such that $\varepsilon
<M_{2}^{\ast }(x,y)<\varepsilon +\delta $ implies $d(T^{m}x,T^{m}y)\leq
\varepsilon $ for a given $\varepsilon >0.$
\end{enumerate}

Then $T$ has a unique fixed point. Also, $T$ is discontinuous at $y_{0}$ if
and only if $\underset{x\rightarrow y_{0}}{\lim }M_{2}^{\ast }(x,y_{0})\neq
0 $.
\end{theorem}

\begin{proof}
Using Theorem \ref{thm2}, we see that the function $T^{m}$ has a unique
fixed point $y_{0}$, that is, $T^{m}y_{0}=y_{0}$. Hence we get%
\begin{equation*}
Ty_{0}=TT^{m}y_{0}=T^{m}Ty_{0}
\end{equation*}%
and so $Ty_{0}$ is a fixed point of $T^{m}$. From the uniqueness of the
fixed point, then we obtain $Ty_{0}=y_{0}$. Consequently, $T$ has a unique
fixed point.
\end{proof}

\begin{remark}
\label{rem2} Using the continuity of the self-mapping $T^{2}$ $($resp. the
continuity of the self-mapping $T^{p}$, the orbitally continuity of the
self-mapping $T)$ and the number $M_{2}(x,y)$, we can also give new
fixed-point results for this open question $($see \cite{Bisht-2017-1} and
\cite{Bisht-2017-2} for this approach$)$.
\end{remark}

\section{\textbf{An Application of the Main Results to Discontinuous
Activation Functions in Neural Networks}}

\label{sec:2} Discontinuous activation functions in neural networks have
been become important and frequently do arise in practise (see \cite%
{Forti-2003} and \cite{Nie-activation} for more details). In this section,
we give an application of the results obtained in Section \ref{sec:1} to
discontinuous activation functions. Recently, this topic has been
extensively studied.

In \cite{Wang-Mexican}, the multistability analysis was investigated for
neural networks with a class of continuous (but not monotonically
increasing) Mexican-hat-type activation functions defined by%
\begin{equation}
T_{i}x=\left\{
\begin{array}{ccc}
m_{i} & ; & -\infty <x<p_{i} \\
l_{i,1}x+c_{i,1} & ; & p_{i}\leq x\leq r_{i} \\
l_{i,2}x+c_{i,2} & ; & r_{i}<x\leq q_{i} \\
m_{i} & ; & q_{i}<x<+\infty%
\end{array}%
\right. \text{,}  \label{eqn15}
\end{equation}%
where $p_{i},$ $r_{i},$ $q_{i},$ $m_{i},$ $l_{i,1},$ $l_{i,2},$ $c_{i,1}$
and $c_{i,2}$ are constants with $-\infty <p_{i}<r_{i}<q_{i}<+\infty $, $%
l_{i,1}>0$ and $l_{i,2}<0$, $i=1,2,\ldots ,n$.

In \cite{Nie-activation}, with the inspiration from the continuous
Mexican-hat-type activation function (\ref{eqn15}), a general class of
discontinuous activation functions was defined by%
\begin{equation}
T_{i}x=\left\{
\begin{array}{ccc}
u_{i} & ; & -\infty <x<p_{i} \\
l_{i,1}x+c_{i,1} & ; & p_{i}\leq x\leq r_{i} \\
l_{i,2}x+c_{i,2} & ; & r_{i}<x\leq q_{i} \\
v_{i} & ; & q_{i}<x<+\infty%
\end{array}%
\right. \text{,}  \label{eqn16}
\end{equation}%
where $p_{i},$ $r_{i},$ $q_{i},$ $u_{i},$ $v_{i},$ $l_{i,1},$ $l_{i,2},$ $%
c_{i,1}$ and $c_{i,2}$ are constants with $-\infty
<p_{i}<r_{i}<q_{i}<+\infty $, $l_{i,1}>0$, $l_{i,2}<0$, $%
u_{i}=l_{i,1}p_{i}+c_{i,1}=l_{i,2}q_{i}+c_{i,2}$, $%
l_{i,1}r_{i}+c_{i,1}=l_{i,2}r_{i}+c_{i,2}$, $v_{i}>T_{i}r_{i}$, $%
i=1,2,\ldots ,n$. It can be easily seen that the function $T_{i}x$ is
continuous in $%
%TCIMACRO{\U{211d} }%
%BeginExpansion
\mathbb{R}
%EndExpansion
$ except the point of discontinuity $x=q_{i}$. Then, it was studied the
problem of multistability of competitive neural networks with discontinuous
activation functions (see \cite{Nie-activation} for more details).

\begin{figure}[h]
\centering
\includegraphics{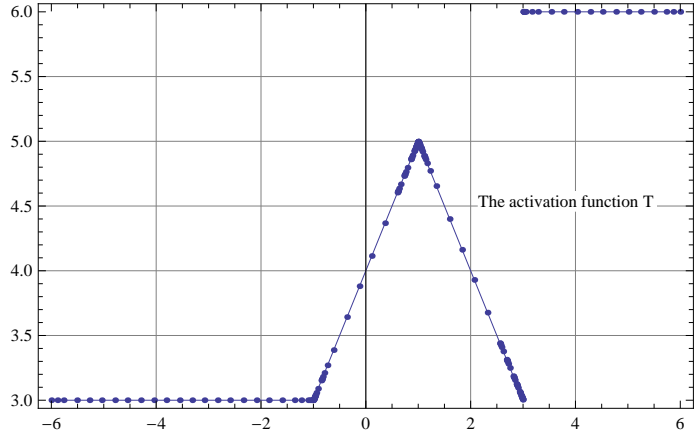}
\caption{The graphic of the discontinuous activation function given in (%
\protect\ref{eqn17}).}
\label{fig:exm1}
\end{figure}

To obtain an application of our results given in the previous section, now
we take
\begin{equation*}
\begin{array}{l}
p_{i}=-1,r_{i}=1,q_{i}=3, \\
u_{i}=3,v_{i}=6,l_{i,1}=1, \\
c_{i,1}=4,l_{i,2}=-1,c_{i,2}=6\text{,}%
\end{array}%
\end{equation*}%
in (\ref{eqn16}) to get the following discontinuous activation function:%
\begin{equation}
Tx=\left\{
\begin{array}{ccc}
3 & ; & -\infty <x<-1 \\
x+4 & ; & -1\leq x\leq 1 \\
-x+6 & ; & 1<x\leq 3 \\
6 & ; & 3<x<+\infty%
\end{array}%
\right. \text{.}  \label{eqn17}
\end{equation}%
The function $Tx$ has two fixed points $x_{1}=3$ and $x_{2}=6$. Since we
have $\underset{x\rightarrow 6}{\lim }M_{1}(x,6)=0$ $\left( \text{resp. }%
\underset{x\rightarrow 6}{\lim }M_{2}(x,6)=0\right) $, $T$ is continuous at
the fixed point $6$. But, there is not a limit of $M_{1}(x,3)$ $\left( \text{%
resp. }M_{2}(x,3)\right) $ as $x\rightarrow 3$ and so $T$ is discontinuous
at the fixed point $3$ $($see Figure \ref{fig:exm1}$)$. Consequently, using
the numbers $M_{1}(x,y)$ and $M_{2}(x,y)$ we can see that the activation
function defined in (\ref{eqn17}) is discontinuous at which fixed points.

More generally, in the case that the number of the fixed points of an
activation function is greater than two, our results will become important
to determine the discontinuity at fixed points. We note that fixed points
can be infinitely many. For example, there are some functions which fix a
circle with infinitely many points and these kind functions can be
considered as activation functions. For example, in \cite{Ozdemir-HNNV} it
was used new types of activation functions which fix a circle for a complex
valued neural network (CVNN). The existence of fixed points of the
complex-valued Hopfield neural network (CVHNN) was guaranteed by using these
types of activation functions. By these reasons, now we consider Theorem \ref%
{thm7} and the number $M_{1}(x,y)$ together. We obtain the following
proposition.

\begin{proposition}
Let $(X,d)$ be a metric space, $T$ be a self-mapping on $X$ and $C_{x_{0},r}$
be a fixed circle of $T$. Then $T$ is discontinuous at any $x\in C_{x_{0},r}$
if and only if $\underset{z\rightarrow x}{\lim }M_{1}(z,x)\neq 0$.
\end{proposition}

\begin{proof}
Let $T$ be a continuous self-mapping at $x\in C_{x_{0},r}$ and $%
x_{n}\rightarrow x$. Then $Tx_{n}\rightarrow Tx=x$ and $d(x_{n},Tx_{n})%
\rightarrow 0$. Hence we get $\underset{n}{\lim }M_{1}(x_{n},x)=0$.

On the other hand, if $\underset{x_{n}\rightarrow x}{\lim }M_{1}(x_{n},x)=0$
then $d(x_{n},Tx_{n})\rightarrow 0$ as $x_{n}\rightarrow x$. This implies $%
Tx_{n}\rightarrow x=Tx$, that is, $T$ is continuous at $x$.
\end{proof}

\begin{example}
If we consider the function $T$ defined in $($\ref{eqn17}$)$ then it is easy
to check that $T$ satisfies the conditions of Theorem \ref{thm7} for the
circle $C_{x_{0},r}=\left\{ 3,6\right\} $ with the center $x_{0}=\frac{9}{2}$
and the radius $r=\frac{3}{2}$. Therefore $T$ fixes the circle $%
C_{x_{0},r}=\left\{ 3,6\right\} $ as another point of view. By the above
proposition, it can be easily deduced that the function $T$ is continuous at
the point $x_{1}=6$ but is not continuous at $x_{2}=3$.
\end{example}

Finally, we note that it is possible to use the number $M_{2}(x,y)$ for the
investigation of discontinuity at any point on the fixed circle of the
activation function.

\section{\textbf{Conclusion}}

We mention that our main results are applicable to neural nets under
suitable conditions (see \cite{Cromme-1991}, \cite{Cromme-1997} and \cite%
{Todd-1976} for more details). For example, McCulloch-Pitts model is
frequently used in Biology and Artificial Intelligence according to the
discontinuity at fixed point. Also our main results can be applied on
complex-valued metric spaces since discontinuity of functions have been used
in various applicable areas such as complex-valued Hopfield neural networks
(see \cite{Wang-Hopfield} for more details).

\textbf{Acknowledgement.} The authors gratefully thank to the Referees for
the constructive comments and recommendations which definitely help to
improve the readability and quality of the paper.


\begin{thebibliography}{99}
\bibitem{Bisht-2017-1} R. K. Bisht and R. P. Pant, A remark on discontinuity
at fixed point, \textit{J. Math. Anal. Appl. }445 (2017) 1239-1242.

\bibitem{Bisht-2017-2} R. K. Bisht and R. P. Pant, Contractive definitions
and discontinuity at fixed point, \textit{Appl. Gen. Topol.} 18 (1) (2017)
173-182.

\bibitem{Chandok-2015} S. Chandok, Some fixed point theorems for $(\alpha
,\beta )$-admissible Geraghty type contractive mappings and related results,
\textit{Math. Sci. }9 (2015) 127-135. doi: 10.1007/s40096-015-0159-4

\bibitem{Cromme-1991} L. J. Cromme and I. Diener, Fixed point theorems for
discontinuous mapping, \textit{Math. Program.} 51 (1991) 257-267.

\bibitem{Cromme-1997} L. J. Cromme, Fixed point theorems for discontinuous
functions and applications, \textit{Nonlinear Analysis: Theory, Methods \&
Applications }30 (3) (1997) 1527-1534.

\bibitem{Forti-2003} M. Forti and P. Nistri, Global convergence of neural
networks with discontinuous neuron activations, \textit{IEEE Trans. Circuits
Syst. I, Fundam. Theory Appl.} 50 (11) (2003) 1421-1435.

\bibitem{Jachymski-1994} J. Jachymski, Common fixed point theorems for some
families of maps, \textit{Indian J. Pure Appl. Math.} 25 (9) (1994) 925-937.

\bibitem{Jachymski-1995} J. Jachymski, Equivalent conditions and Meir-Keeler
type theorems, \textit{J. Math. Anal. Appl.} 194 (1995) 293-303.

\bibitem{Kannan-1969} R. Kannan, Some results on fixed points--II, \textit{%
Amer. Math. Monthly} 76 (1969) 405-408.

\bibitem{Nie-activation} X. Nie and W. X. Zheng, On multistability of
competitive neural networks with discontinuous activation functions, Control
Conference (AUCC), 2014 4th Australian{\ }245-250.

\bibitem{Ozdemir-HNNV} N. \"{O}zdemir, B. B. \.{I}skender and N. Y. \"{O}zg%
\"{u}r, Complex valued neural network with M\"{o}bius activation function,
\textit{Commun. Nonlinear Sci. Numer. Simul.} 16 (12) (2011) 4698-4703.

\bibitem{Ozgur-circle} N. Y. \"{O}zg\"{u}r and N. Ta\c{s}, Some fixed-circle
theorems on metric spaces, Bull. Malays. Math. Sci. Soc. (2017).
https://doi.org/10.1007/s40840-017-0555-z

\bibitem{Ozgur-circle-S1} N. Y. \"{O}zg\"{u}r and N. Ta\c{s}, Some
fixed-circle theorems on $S$-metric spaces with a geometric viewpoint,
arXiv:1704.08838 [math.MG].

\bibitem{Ozgur-circle-S2} N. Y. \"{O}zg\"{u}r, N. Ta\c{s} and U. \c{C}elik,
Some fixed-circle results on $S$-metric spaces, Bull. Math. Anal. Appl. 9
(2) (2017) 10-23.

\bibitem{Pant-1999} R. P. Pant, Discontinuity and fixed points, \textit{J.
Math. Anal. Appl. }240 (1999) 284-289.

\bibitem{Rhoades-1977} B. E. Rhoades, A comparison of various definitions of
contractive mappings, \textit{Trans. Amer. Math. Soc.} 226 (1977) 257-290.

\bibitem{Rhoades-1988} B. E. Rhoades, Contractive definitions and
continuity, \textit{Contemp. Math.} 72 (1988) 233-245.

\bibitem{Todd-1976} M. J. Todd, The computation of fixed points and
applications, Springer-Verlag, Berlin, Heidelberg, New York, 1976.

\bibitem{Wang-Mexican} L. L. Wang and T. P. Chen, Multistability of neural
networks with Mexican-hat-type activation functions, \textit{IEEE Trans.
Neural Netw. Learn. Syst.} 23 (11) (2012) 1816-1826.

\bibitem{Wang-Hopfield} Z. Wang, Z. Guo, L. Huang and X. Liu, Dynamical
behavior of complex-valued Hopfield neural networks with discontinuous
activation functions, \textit{Neural Process Lett} (2016).
doi:10.1007/s11063-016-9563-5
\end{thebibliography}
\end{document}